\documentclass{amsart}
\usepackage{amssymb}
\usepackage{amscd}
\input xy
\xyoption{all}
\newtheorem{theorem}{Theorem}[section]
\newtheorem{lemma}[theorem]{Lemma}

\newtheorem{proposition}[theorem]{Proposition}
\theoremstyle{definition}
\newtheorem{definition}[theorem]{Definition}
\newtheorem{example}[theorem]{Example}

\theoremstyle{remark}
\newtheorem{remark}[theorem]{Remark}

\numberwithin{equation}{section}

\begin{document}

\title{Torsors and the Quillen-Barr-Beck cohomology for restricted Lie algebras}

\author{Ioannis Dokas}
%\address{Department of Mathematics, Louisiana State University, Baton
%Rouge, Louisiana 70803}
%\curraddr{Department of Mathematics and Statistics,
%Case Western Reserve University, Cleveland, Ohio 43403}
%\email{xyz@math.university.edu}

%\author{Author Two}
%\address{Mathematical Research Section, School of Mathematical Sciences,
%Australian National University, Canberra ACT 2601, Australia}
\email{yiannis.dokas@gmail.com}
\thanks{I would like to express my appreciation to the Isaac Newton Institute for the excellent working conditions that I enjoyed there during the programme
``Grothendieck-Teichm\``{u}ller groups, Deformations and Operads''.}

%    General info

\date{}

%\dedicatory{This paper is dedicated }


\begin{abstract}
Given an exact sequence of restricted Lie algebras using Duskin's
torsors theory, we establish an eight term exact sequence for
Quillen-Barr-Beck cohomology of restricted Lie algebras. As an
application we obtain for any extension of algebraic groups over an
algebraic closed field of prime characteristic an eight term exact
sequence for the corresponding restricted Lie algebras extension.
\end{abstract}

\maketitle

%% The correct journal style for \specialsection is all uppercase; a known bug
%% in amsart.cls prevents this, so input must be uppercase until it is fixed.
%\specialsection*{This is a Special Section Head}

%%%%%%%%%%%%%%%%%%%%%%%%%%%%%%%%%%%%%%%%%%%%%%%%%%%%%%%%%%%%%%%%%%%%%%%%
%\footnote{Here is an example of a footnote. Notice that this
%footnote text is running on so that it can stand as an example of
%how a footnote with separate paragraphs should be written.
%\par
%And here is the beginning of the second paragraph.}%
%%%%%%%%%%%%%%%%%%%%%%%%%%%%%%%%%%%%%%%%%%%%%%%%%%%%%%%%%%%%%%%%%%%%%%%%
\section{Introduction}

Let $0\rightarrow H \rightarrow G \rightarrow  A\rightarrow 0$ be an
extension of commutative algebraic groups over an algebraic closed
field. If $B$ is another
 commutative algebraic group, then J.-P. Serre in \cite{Ser} and M. Rosenlicht in \cite{Ros} proved that there is an exact  sequence
\begin{multline*}
0\rightarrow Hom(A,B) \rightarrow Hom(G,B) \rightarrow Hom(H,B)\rightarrow\\
\rightarrow Ext(A,B) \rightarrow Ext(G,B)\rightarrow Ext(H,B)
\end{multline*}

If the ground field has prime  characteristic then the extension of
commutative algebraic groups induces a short exact sequence of
abelian restricted Lie algebras. In this paper we associate to this
exact sequence of abelian restricted algebras, an eight term exact
sequence for the Quillen-Barr-Beck cohomology of restricted Lie
algebras. Using Duskin's torsors theory we construct in Theorem
$4.4$, for any short exact sequence $0\rightarrow N \rightarrow
\mathbf{g} \rightarrow \mathbf{b} \rightarrow 0$ of restricted Lie
algebras, an eight term exact sequence for Quillen-Barr-Beck
cohomology with coefficients any Beck $\mathbf{b}$-module.

Here we use the cohomology defined in \cite{dokas} following the
general scheme of Quillen-Barr-Beck cohomology theory for universal
algebras (see \cite{BB}, \cite{Quillen2}, \cite{Quillen}). In
contrast to the cases of Groups, Associative algebras, Lie algebras,
where we obtain the classical cohomologies theories, we do not
obtain Hochschild (co)-homology  for restricted Lie algebras
\cite{Hoch}. We obtain different cohomology which classifies more
general abelian extensions of restricted Lie algebras and not
strongly abelian extensions which are classified by Hochschild
cohomology.

The Hochschild-Serre spectral sequence for (co)-homology in various
algebraic categories  gives rise to exact sequences of terms of low
degree. For the category of restricted Lie algebras it is proved by
Eckmann and Stammbach in \cite{ES} that there is a five-term exact
sequence for Hochschild (co)-homology  of restricted Lie algebras.
In Quillen-Barr-Beck cohomology this is replaced by the eight-term
exact sequence of Theorem $4.4$.

We now give an outline of the methods used in our proof. Duskin (cf.
\cite{duskin}), develops the theory of $n$-torsors in order to give
an interpretation of the $n$-th cotriple cohomology group. Later,
Glenn in \cite{Glenn} defines torsor cohomology groups and gives a
slightly different definition of $n$-th torsor which coincides with
the one given by Duskin in the context of a monadic category over
$\rm{Sets}$ whose objects have an underlying group structure (cf.
Corollary $7.2.4$ in \cite{Glenn}). Cegarra and Aznar in
\cite{CegAz} proved  in the framework of a Barr-exact category that
there is an eight term exact sequence for torsor cohomology. We use
this result in order to obtain exact sequences for Quillen-Barr-Beck
cohomology for restricted Lie algebras.

In Lemma $2.1$ of Section $2$ we compute the $0$-torsors group for
the category of restricted Lie algebras. Thus, we obtain a $5$-term
exact sequence for Quillen-Barr-Beck cohomology of restricted Lie
algebras.

In section $3$ in order to study $2$-torsors groups we are led to
consider crossed modules and internal groupoids in the category of
restricted Lie algebras. In Theorem $3.3$ is obtained that the
notions of crossed modules and internal groupoids are equivalent.

In section $4$ we study $2$-fold extensions for the category of
restricted Lie algebras. Two fold extensions in certain categories
of interest are studied by M. Gerstenhaber, in \cite{Ger} and is
given cohomological interpretation. Besides, J.-L. Loday and C.
Kassel in \cite{Loday2} consider $2$-fold extensions of Lie algebras
in connection to the study of relative cohomology of Lie algebras.
In Lemma $4.2$ we prove that the $2$-torsors cohomology group is
isomorphic to the group of equivalence classes of $2$-fold
extensions of restricted Lie algebras. Thus we obtain an
interpretation of the second Quillen-Barr-Beck cohomology group in
terms of $2$-fold extensions. As a consequence using Cegarra-Aznar's
eight term exact sequence we establish in Theorem $4.4$ an
eight-term exact sequence for Quilen-Barr-Beck cohomology of
restricted Lie algebras which extends the $5$-term exact sequence of
Section $2$. In the last subsection $4.4$ we apply Theorem $4.4$ to
extensions of algebraic groups in prime characteristic.

\subsection{Restricted Lie algebras}

In modular Lie theory in order to extent theorems which are valid in
characteristic zero we are led to consider restricted  Lie algebras
introduced by N. Jacobson in \cite{Jac}. Let $k$ denote a field of characteristic $p\neq 0$ and $\rm{Lie}$ the category of Lie algebras over $k$.
\begin{definition}
A \textit{restricted Lie algebra} $(L,(-)^{[p_{L}]})$ over $k$ is a
Lie algebra $L \in \rm{Lie}$ together with a map $(-)^{[p_{L}]}: L
\rightarrow L$ called the $p$-map such that the following relations
hold:

\begin{align}
(\alpha x)^{[p_{L}]} &=  \alpha^{p}\;x^{[p_{L}]}\\
[x,y^{[p_{L}]}]     &= [x,\underbrace{y],y],\cdots,y}_{p}] \\
(x+y)^{[p_{L}]} &=
x^{[p_{L}]}+y^{[p_{L}]}+\sum_{i=1}^{p-1}s_{i}(x,y)
\end{align}
where $i s_{i}(x,y)$ is the coefficient of $\lambda^{i-1}$ in
$ad^{p-1}_{\lambda x+y}(x)$, where $ad_{x}: L\rightarrow L$ denotes
the adjoint representation given by $ad_{x}(y):=[y,x]$ and $x,y \in
L,\; \alpha \in k$. A Lie algebra homomorphism $f: L\rightarrow L'$
is called restricted if $f(x^{[p]})=f(x)^{[p]}$. We denote by
$\rm{RLie}$ the category of restricted Lie algebras over $k$.

\begin{remark}
Let $L \in \rm{RLie}$ be a restricted Lie algebra and $x,y\in L$. If $L_{x,y}$ is the Lie algebra generated by the elements $x,y$ then $s_{i}(x,y)\in L_{x,y}^{p}$ where $L^{1}_{x,y}:=L_{x,y}$ and $$L_{x,y}^{n}:=[L_{x,y}^{n-1},L_{x,y}]$$ for  $0\leq  i \leq p-1$.
\end{remark}

\begin{remark}
There is a notion of free restricted Lie algebra over a set.
Therefore, the category $\rm{RLie}$ of restricted Lie algebras is a
monadic category over $\rm{Sets}$. It follows that $\rm{RLie}$ is a
Barr exact category.
\end{remark}

Let $L$ and $L'$ be restricted Lie algebras. The direct product of $L$ and $L'$ is as a Lie algebra their direct product in the category
of Lie algebras $$L\times L'= \{ (x,y), : x\in L, y\in L' \}$$  equipped with the $p$-map given by $(x,y)^{[p]}:=(x^{[p]},y^{[p]})$ for all $(x,y)\in L\times L'$. If $f: L\rightarrow R$ and $f':L'\rightarrow R$ are restricted Lie homomorphisms, then the pullback of $L$ and $L'$ over $R$ is given by the following commutative diagram in $\rm{RLie}$
\[
\begin{CD}
  L\times_{R} L' @>pr'>> L' \\
  @VVprV @VVf'V \\
  L @>f>> R
\end{CD}
\]
where $$L\times_{R}L':=\{(x,y)\in L\times L', : f(x)=f'(y)\}$$

Let $L\in \rm{RLie}$ be a restricted Lie algebra and $U(L)$ its
enveloping algebra. We denote by $u(L):=U(L)/<x^{p}-x^{[p]}>,\;x\in
L$ the restricted enveloping algebra of $L$. A $L$-module $A$ is
called \textit{restricted} if $x^{[p]}a=(\underbrace{x\cdots
(x(x}_{p}a)\cdots)$. The category of restricted Lie $L$-modules is
equivalent to the category of $u(L)$-modules.
\end{definition}

\begin{example}
Let $\mathfrak{A}$ be any associative algebra over a field $k$ with
characteristic $p\neq 0$. We denote by $\mathfrak{A}_{Lie}$ the induced Lie
algebra with the bracket given by $[x,y]:=xy-yx$, for all $x,y\in
\frak{A}$. Then $(\mathfrak{A},(-)^{p})$ is a restricted Lie algebra where $(-)^{p}$
is the Frobenious map given by $x\mapsto x^{p}$. Thus, there is a
functor $(-)_{\rm{RLie}}: \rm{As} \rightarrow \rm{RLie}$ from the
category of associative algebras to the category of restricted Lie
algebras.
\end{example}

\begin{example}
Let $\mathfrak{A}$ be an associative algebra over $k$. Then
$\mathfrak{gl}_{n}(\mathfrak{A})$ the Lie algebra of $n\times n$ matrices with
coefficients in $\mathfrak{A}$ is a restricted Lie algebra. Moreover, one
proves that $\mathfrak{sl}_{n}(\mathfrak{A})$ is a restricted Lie subalgebra of
$\mathfrak{gl}_{n}(\mathfrak{A})$.
\end{example}

\begin{example}
Let $V$ be a $k$-vector space then a map $f: V\rightarrow V$ such that $f(x+y)=f(x)+f(y)$ and $f(\alpha x)=\alpha^{p}f(x)$ for all $x,y \in V$ and $\alpha \in k$
is called $p$-semi-linear. Any pair $(V,f)$ where $V$ is $k$-vector space and $f: V\rightarrow V$  a $p$-semi-linear map is an abelian restricted Lie algebra.
\end{example}

\begin{example}
Let $\mathfrak{B}$ be an $k$-algebra not necessarily associative
then the set of $k$-derivations $Der(\mathfrak{B})$ is endowed with
the structure of restricted Lie algebra. In particular, if $D\in
Der(\mathfrak{B})$ by the Leibniz formula we one has
$$D^{p}(xy)=\sum_{i=0}^{i=p}\binom{p}{i}D^{i}(x)D^{p-i}(y)$$ for all
$x,y\in \mathfrak{B}$. Since the $char\;k=p$ we have
$\binom{p}{i}=0$ for $1\leq i \leq p-1$, therefore $D^{p}\in
Der(\mathfrak{B})$.
\end{example}
Let $(L,[p]) \in \rm{RLie}$ be a restricted Lie algebra then a derivation $D\in Der(L)$ is called restricted
derivation if $D(x^{[p]})=ad^{p-1}(x)(D(x))$ for all $x \in L$.

\begin{example}
Let $\rm{G}$ be an algebraic group over $k$. The associated Lie
algebra $Lie(\rm{G})$ of $\rm{G}$ is endowed with the structure of
restricted Lie algebra (cf. \cite{Bor}).
\end{example}

Let $(L,(-)^{[p]}), (N,(-)^{[p']})$ be two restricted Lie algebras, and $\eta: L \rightarrow Der(N)$ a restricted homomorphism such that $\eta (x)$ is a restricted derivation for every $x\in L$. We recall that the Lie product of the semi-direct product of $L$ and $N$ is given by
$$[(l,n),(l',n')]=([l,l'],\eta(l)n'-\eta(l')n+[n,n'])$$

 Then it follows from Jacobson's Theorem $11$ in \cite{Jac} that the semi-direct product of $L$ and $N$ is endowed with the a $p$-map extending the $p$-maps on $L$ and $N$ (cf. Theorem $2.5$ in \cite{Far}). We call this restricted Lie algebra the \textit{semi-direct product} of $L$ and $N$ in the category of restricted Lie algebras and we denote it by $L\rtimes N$.

Let us recall some definitions and results which we use in the next
section. We refer the reader to \cite{dokas} for details.

In his doctoral dissertation, Beck incorporates the various notions
of  module over enveloping algebras to the general notion of a Beck
module.

\subsection{Beck-modules for \rm{RLie}}
Let $L\in \rm{RLie}$ be a restricted Lie algebra, $A$ a restricted
$L$-module and $f: A \rightarrow A^{L}$ a $p$-semi-linear map from
$A$ to the invariants $A^{L}=\{a\in M : \text{$xa=0$ for all $x\in
L$}\}$. We denote by $L \times_{f} A$ the semi-direct product in
$\rm{RLie}$ of $L$ and $A$. In particular, $L \times_{f} A$ is the
semi-direct product in $Lie$ together with the $p$-map
$(l,a)^{[p]}=(l^{[p]},\underbrace{l\cdots l}_{p-1}a+f(a))$ where
  $x\in L$ and $a\in A$. For $L \in \rm{RLie}$ a restricted Lie algebra the next theorem
gives a characterization of the category of abelian group objects of
the slice category $ab(\rm{RLie}/L)$ i.e the category Beck
$L$-modules in $\rm{RLie}$.

\begin{theorem}
The category of abelian group objects $ab(\rm{RLie}/L)$ is
equivalent to the category $\mathcal{A}$ whose objects are pairs
$(A,f)$ where $A$ is a restricted $L$-module and $f: A \rightarrow
A^{L}$ is a $p$-semi-linear map from $A$ into its submodule of
invariants $A^{L}$ and whose morphisms $(A_{1},f_{1}) \rightarrow
(A_{2},f_{2})$ are $L$ homomorphisms $\alpha : A_{1} \rightarrow
A_{2}$ such that $f_{2} \alpha = \alpha f_{1}.$
\end{theorem}

\begin{proof}
See Theorem $1.7$ in \cite{dokas}.
\end{proof}

Let $R_{f}$ be the polynomial ring consisting of the set of
polynomials $\sum_{i=0}^{i=m}a_{i}f^{i}$ where $a_{i} \in k$, $f$ an
indeterminate and $fa=a^{p}f$. We denote by $w(L)$ the ring which as
$k$-vector space is $w(L):=R_{f} \otimes _{k} u(L)$ and such that
$R_{f}\rightarrow w(L)$ and $u_{L}\rightarrow w(L)$ are algebra
homomorphisms and
$$(P\otimes 1)(1\otimes l):=P\otimes l \;\text{and}\; (l\otimes 1) (P\otimes 1):=0$$
for all $P\in R_{f}$ and $l\in L$. The following theorem gives
another characterization of the category of abelian group objects
$ab(\rm{RLie}/L)$.

\begin{theorem}
The category of abelian group objects
$ab(\rm{RLie}/\boldsymbol{L})$ is equivalent to the category of
$w(\boldsymbol{L})$-modules.
\end{theorem}

\begin{proof}
See Theorem $1.8$ in \cite{dokas}.
\end{proof}

\begin{remark}
By Theorems $1.9$, $1.10$ we note that if $L\in \rm{RLie}$, then
each $w(L)$-module $A$ is associated to a couple $(\bar{A},f)$ where
$\bar{A}$ is a restricted $L$-module and $f$ a $p$-semi-linear map
such that $l.f(\bar{a})=0$ for all $l \in L$ and $\bar{a} \in
\bar{A}$.
\end{remark}

\subsection{Beck-derivations for \rm{RLie}} Let $g\in RLie/L$ be a restricted Lie
algebra and  $A$ an $w(L)$-module. The group of Beck derivations is
defined as follows
$$Der_{p}(g,A):=\{d\in Der(g,\bar{A}):
d(x^{[p]})=\underbrace{x\cdots x}_{p-1}dx+f(d(x)),\;x\in g\}$$

It is proved in \cite{dokas} the isomorphism
\begin{equation}
Hom_{\rm{RLie}/L}(\emph{g}, L\times_{f}\bar{A})\simeq
Der_{p}(\emph{g},A)
\end{equation}
given by $$f\mapsto pr_{A}f$$ where $pr_{A}$ denotes the canonical projection.

Cartan and Eilenberg in their book \cite{CE} describe a general
context for the definition of (co)-homology groups for various
algebraic structures. In each case is used an appropriate notion of
enveloping algebra and the definition is given in terms of $Ext$ and
$Tor$ functors. In this context Hochschild in \cite{Hoch} and B.
Pareigis in \cite{Pa} defined (co)-homology groups for the category
of restricted Lie algebras. In the paper of Barr and Beck \cite{BB} is given the
relation between cotriple cohomology groups and the cohomology
theories described in the Cartan-Eilenberg context. It is proved
that for the categories of Groups, Associative algebras, Lie
algebras the two theories coincide (considering a shift in
dimension). Also, cohomology as a derived functor of derivations are studied for
several algebraic categories by Barr and Rinehart in \cite{BarRin}. Besides,
D. Quillen in \cite{Quillen2}, \cite{Quillen} develops an axiomatic homotopy theory and
cohomology groups are defined in the context of model categories. Through this development are
defined (co)-homology groups for universal algebras (see section $2$
in \cite{Quillen}). In \cite{dokas} is defined Quillen-Barr-Beck cohomology
for the category of restricted Lie algebras.

\subsection{Quillen-Barr-Beck cohomology for \rm{RLie}}

Let $L\in \rm{RLie}$ be a restricted Lie algebra and $A$ a $w(L)$-module. Let $F: Sets \rightarrow
\rm{RLie}$ be the free functor, left adjoint to the forgetful
functor $U: \rm{RLie} \rightarrow Sets$. The adjunction $(F,U)$ give
rise to a cotriple $\textbf{G}$ in $\rm{RLie/L}$. In \cite{dokas} are
defined cotriple cohomology groups

 $$H_{\textbf{G}}^{n}(L,A):= H^{n}(Der_{p}(\textbf{G}_{*}(L),A))$$

Bellow we recall the definition of a $n$-torsor, for details and terminology we refer the reader to \cite{duskin}, \cite{duskin2} and \cite{Glenn}.

\subsection{Torsors, interpretation of cotriple cohomology}

Duskin in \cite{duskin} gave an interpretation of the cotriple
cohomology in terms of $n$-dimensional torsors generalizing to any
dimension Beck's interpretation of dimension $1$. Let $\rm{E}$ be a
monadic category over Sets and $\mathbf{G}=FU$ the associated
cotriple. The $n$-truncating functor
$$tr^{n}:Simpl(\rm{E})\rightarrow tr^{n}Simpl(\rm{E})$$ from the
category of simplicial objects of $\rm{E}$ to the category of
$n$-truncated simplicial objects admits a right adjoint $$cosk^{n}:
tr^{n}Simpl(\rm{E}) \rightarrow Simpl(\rm{E})$$ called Verdier's
coskeleton functor. The coskeleton functor is constructed by
iterating simplicial kernels. We denote by $Cosk^{n}$ the
composition functor $$Cosk^{n}:=cosk^{n}\circ tr^{n}: Simpl(\rm{E})
\rightarrow Simpl(\rm{E})$$

Let $A \in ab(\rm{E})$ be an abelian group object, then  the simplicial object $K(A,n)$ is defined as the
$(n+1)$-coskeleton of the following  $(n+1)$-truncated simplicial object

\bigskip

\xymatrix{ A^{n+1} \ar@<-7pt>[r]_-{pr_{0}}^{:} \ar@<0.5pt>[r]^-{pr_{n}}
\ar@<9pt>[r]^-{k_{n}}& A  \ar[r] &
\mathbf{1}  \ar[r] &\cdots  \ar[r]&  \mathbf{1}}
where $\mathbf{1}$ denotes the terminal object and $k_{n}=(-1)^{n} \sum_{i=0}^{i=n}(-1)^{i}pr_{i}$ for all $n\geq 1$.

\begin{definition}
Let $X \in \rm{E}$, then a  $K(A,n)$-torsor in $\rm{E}$ over $X$ relative to $U$ is  defined as an augmented over $X$ simplicial object $(X_{.},d_{i},s_{i})$ together with a simplicial morphism $\chi: X_{.} \rightarrow K(A,n)$  such that
\begin{enumerate}
  \item $X_{.} \rightarrow X$ is a $U$-split augmented simplicial object

  \item   $\chi: X_{.} \rightarrow K(A,n)$ is a simplicial morphism such that the following squares are pullbacks, for each $m \geq n$ and all $0 \leq i \leq m$
 \[
      \begin{CD}
X_{m} @>\chi_{m}>> K(A,n)_{m}\\
@VV(d_{0},\cdots ,\hat{d}_{i},\cdots ,d_{m})V @VV(k_{0},\cdots ,\hat{k}_{i},\cdots ,k_{m})V\\
\Lambda^{i}(m-1)(X_{.}) @>>> \Lambda^{i}(m-1)(K(A,n)_{.})
\end{CD}
\]
where $\Lambda^{i}(m-1)$ denotes the $i$-horn simplicial object of the $(m-1)$-truncated simplicial object
  \item the canonical map $X_{.} \rightarrow Cos^{n-1}_{aug}(X_{.})$ is an isomorphism, where $Cosk^{n}_{aug}$ denotes augmented $n$-coskeleton functor.
\end{enumerate}
\end{definition}

If $(X,s_{.},\chi)$  and $(X',s_{.},\chi')$ are two $K(A,n)$-torsors over $X$ then a morphism of $K(A,n)$-torsors is an $X$-map  $f_{.}: X_{.} \rightarrow X'_{.}$  of augmented simplicial objects such that $\chi=\chi'f$. The set of the
connected components of $n$-torsors is denoted by $Tors^{n}(X,A)$. Let $H_{\mathbf{G}}^{*}(X,A)$
be the cotriple cohomology groups, then Duskin in \cite{duskin}
proved the following theorem.

\begin{theorem}
If $X \in \rm{E}$ and $A$ an abelian group object then
there is a bijection between the set $Tors^{n}(X,A)$ of equivalence
classes of $n$-torsors and the $n$-th cotriple cohomology
$H_{\mathbf{G}}^{n}(X,A)$ for $n\geq 1$.
\end{theorem}
\begin{proof}
Theorem (8.9) in \cite{duskin}.
\end{proof}
Later Glenn in \cite{Glenn} defined the notion of $n$-dimensional
hypergroupoids and gave a slightly different definition of
$n$-torsor. If $E$ is a monadic category over $Sets$ whose objects
have an underlying group structure, then by Corollary $7.2.4$ in
\cite{Glenn} the two notions of $n$-torsors coincide. Let $p: B
\rightarrow R$ be a regular  epimorphism in $\rm{E}$ and $B_{*}^{p}$
the simplicial object (with augmentation $p$), obtained by iterating
the simplicial kernel construction. Then, A.M Cegarra and E. R.
Aznar in \cite{CegAz} define (Definition $1.1$) the abelian groups
$$Tor^{0}(p,A):=Hom_{simpl(\rm{E})}({B}_{*}^{p}, K(A,1))$$
and $Tors^{1}(p,A)$ of connected components of $2$-torsors over $R$
with fixed augmentation $p$ (Definition $2.1$). Moreover, it is
proved in \cite[Theorem~$7.3$]{CegAz} that there is an eight term
exact sequence

\begin{multline}
Hom_{\rm{E}}(R,A)\xrightarrow{p^{*}} Hom_{\rm{E}}(B,A)\rightarrow
Tor^{0}(p,A)\rightarrow Tors^{1}(R,A)\rightarrow
Tors^{1}(B,A)\\
\rightarrow Tors^{1}(p,A)\rightarrow Tors^{2}(R,A)\rightarrow
Tors^{2}(B,A).
\end{multline}

\section{A 5-term exact sequence for Quillen-Barr-Beck cohomology}

The Hochschild-Serre spectral sequence for (co)-homology of Lie
algebras gives rise to exact sequences of terms of low degree. In
particular, if $$0\rightarrow N \rightarrow \mathbf{g} \rightarrow
\mathbf{b}\rightarrow 0$$ is an exact sequence of Lie algebras and
if $A$, is a left (resp. right) $\mathbf{b}$-module, then we have
the following exact sequences
$$
0 \rightarrow Der(\mathbf{b},A)\rightarrow Der(\mathbf{g},A)
\rightarrow Hom_{U(\mathbf{g})}(N_{ab}, A)\rightarrow
H^{2}(\mathbf{b},A)\rightarrow H^{2}(\mathbf{g},A)$$ and
$$H_{2}(\mathbf{g},A)\rightarrow
H_{2}(\mathbf{b},A)\rightarrow N_{ab}
\otimes_{U(\mathbf{b})}A\rightarrow H_{1}(\mathbf{g},A) \rightarrow
H_{1}(\mathbf{b},A)\rightarrow 0$$

where $N_{ab}:=N/[N,N]$.

For the case of Hochschild (co)-homology of restricted Lie algebras,
there are analogue sequences. Precisely, if
$$0\rightarrow N\rightarrow \mathbf{g} \rightarrow
\mathbf{b}\rightarrow 0$$ is an exact sequence of restricted Lie
algebras and if $M$ is a right restricted $\mathbf{b}$-module then
B. Eckmann and  U. Stammbach proved in \cite{ES} the existence of
the following  $5$-term exact sequence
$$H_{2}^{Hoch}(\mathbf{g},A)\rightarrow
H_{2}^{Hoch}(\mathbf{b},A)\rightarrow \mathcal{N}_{ab}
\otimes_{u(\mathbf{b})}A\rightarrow  H_{1}^{Hoch}(\mathbf{g},A)
\rightarrow H_{1}^{Hoch}(\mathbf{b},A)\rightarrow 0$$ where
$\mathcal{N}_{ab}:=N/[N,N]'$ and $[N,N]'$
denotes the ideal generated by the elements
$[x,y],z^{[p]}$ where $x,y,z\in N$.

If $N$ is a restricted Lie algebra we denote by
$\mathbf{N}_{ab}:=N/<[N,N]>_{p}$ the quotient restricted Lie
algebra, where  $<[N,N]>_{p}$ is the $p$-ideal generated by the
elements $[x,y]$ where $x,y\in N$. If $p: \textbf{g} \rightarrow L$ is a
restricted Lie epimorphism with kernel $N$ then the $p$-map on $N$
induce an $R_{f}$-action on $\mathbf{N}_{ab}:=N/<[N,N]>_{p}$ given
by $f.(n+<[N,N]>_{p}):=n^{[p]}+<[N,N]>_{p}$. Besides,
$\mathbf{N}_{ab}$ is a $L$-module via the action
$x.(n+<[N,N]>_{p}):=[s(x),n]+<[N,N]>_{p}$ where $s: L\rightarrow \textbf{g}$ denotes a section of $p$ and $x\in L$, $n\in
N$. Moreover,
\begin{align*}
x.(n^{[p]}+<N,N>_{p})&=[s(x),n^{[p]}]+<[N,N]>_{p}\\
                     &=[s(x),\underbrace{n],\cdots n}_{p}]+<[N,N]>_{p}\\
                     &=0
\end{align*}
It follows that, $\mathbf{N}_{ab}$ is equipped with the structure of
a $w(L)$-module.

\begin{lemma}
If $p: \textbf{g} \rightarrow L$ is a restricted Lie epimorphism
with kernel $N$ and $A$ a in $w(L)$-module, then we have an
isomorphism
$$Tor^{0}(p,A) \simeq Hom_{w(L)}(\textbf{N}_{ab},A)$$
\end{lemma}

\begin{proof}

Let $\mathbf{g}_{*}$ be the simplicial restricted Lie algebra which is obtained by iterating the
simplicial kernel construction i.e

\bigskip
\xymatrix{\textbf{g}_{*}:\;\;\;\ldots \mathbf{g}\times_{L}\mathbf{g}
\times_{L}\mathbf{g} \ar@<-4.5pt>[r]_-{d_{0}} \ar@<-2pt>[r]
\ar@<1pt>[r]^-{d_{2}}&\mathbf{g} \times_{L}\mathbf{g} \ar@/^-1pc/[l]
\ar@/^-1.5pc/[l] \ar@<-4.5pt>[r]_-{d_{0}} \ar@<1pt>[r]^-{d_{1}}&
\mathbf{g} \ar@/^-1pc/[l] \ar[r]&  L}

If we denote by $\textbf{g}\rtimes N$ the semi-direct product of
$\textbf{g}$ and $N$ in $\rm{RLie}$, we get an isomorphism of
restricted Lie algebras $$\textbf{g}\rtimes N \simeq
\textbf{g}\times_{L}\textbf{g}$$ given by $$(x,n)\mapsto (x,x+n)$$

Therefore we get the following simplicial object

\bigskip

\xymatrix{\textbf{g}_{*}:\;\;\;\ldots \textbf{g}\rtimes N \rtimes N \ar@<-4.5pt>[r]_-{d_{0}} \ar@<-2pt>[r]
\ar@<1pt>[r]^-{d_{2}}& \textbf{g}\rtimes N  \ar@/^-1pc/[l]
\ar@/^-1.5pc/[l] \ar@<-4.5pt>[r]_-{d_{0}} \ar@<1pt>[r]^-{d_{1}}&
\mathbf{g} \ar@/^-1pc/[l] \ar[r]&  L}

where
\begin{align*}
d_{0}(x,n,n'): &=(x,n)\\
d_{1}(x,n,n'): &=(x,n')\\
d_{2}(x,n,n'): &=(x+n,n'-n)
\end{align*}

The abelian  group $Tor^{0}(p,A)$ is defined by
$$Tor^{0}(p,A):=Hom_{simpl(\rm{RLie/L})}(\textbf{g}_{*}, K(L\times_{f}\bar{A},1))$$
and by Lemma $2.1$ in \cite{CegAz} we obtain

$$Tor(p,A)=ker\big(Hom_{\rm{RLie}/L}(\textbf{g}\rtimes N,L\times_{f}\bar{A}) \xrightarrow{d_{0}^{*}-d_{1}^{*}+d_{2}^{*}}
Hom_{\rm{RLie}/L}(\textbf{g}\rtimes N\rtimes
N,L\times_{f}\bar{A})\big)$$

Let $\phi\in Tor(p,A)$ then by isomorphism $(1.4)$ any $\phi\in Tor(p,A)$ is associated to a
Beck derivation $d_{\phi}\in Der_{p}(\textbf{g}\rtimes N,A)$ such
that $d_{\phi}d_{0}-d_{\phi}d_{1}+d_{\phi}d_{2}=0$, i.e

\begin{equation}
d_{\phi}(x,n)-d_{\phi}(x,n')+d_{\phi}(x+n,n'-n)=0
\end{equation}

for all $x\in \textbf{g}$ and $n,n'\in N$. Since $N=ker\,p$ we get
\begin{align*}
d_{\phi}(0,[n,n']) &=d_{\phi}\big([(0,n),(0,n')]\big)\\
                   &=(0,n)d_{\phi}\big((0,n')\big)-(0,n')d_{\phi}\big((0,n)\big)\\
                   &=0
\end{align*}
Besides,
\begin{align*}
d_{\phi}(0,[n,n'])^{[p]}) &= d_{\phi}\big((0,[n,n'])^{[p]}\big)\\
                     &=(0\underbrace{,[n,n'])\cdots (0,}_{p-1}[n,n'])d_{\phi}(0,[n,n'])+f\big(d_{\phi}(0,[n,n'])\big)\\
                     &=0
\end{align*}
Therefore is defined a map $\bar{\phi}: \mathbf{N}_{ab} \rightarrow A$ given
by $$\bar{\phi}(n+<[N,N]>_{p}):=d_{\phi}(0,n)$$ for all $n\in N$.
Moreover for $x\in L$ and $n\in N$ we have
\begin{align*}
\bar{\phi}\big((x.(n+<[N,N]>_{p})\big) &=\bar{\phi}([s(x),n]+<[N,N]>_{p})\\
                            &=d_{\phi}(0,[s(x),n])\\
                            &=d_{\phi}\big([(s(x),0),(0,n)]\big)\\
                            &=(s(x),0)d_{\phi}(0,n)-(0,n)d_{\phi}(s(x),0)\\
                            &=(s(x),0)d_{\phi}(0,n)\\
                            &=x.\bar{\phi}(n+<[N,N]>_{p})
\end{align*}

\begin{align*}
\bar{\phi}\big((n^{[p]}+<[N,N]>_{p})\big)&=d_{\phi}(0,n^{[p]})\\
                                         &=d_{\phi}\big((0,n)^{[p]}\big)\\
                                         &=\underbrace{(0,n)\cdots
                                         (0,n)}_{p-1}d_{\phi}(0,n)+f(d_{\phi}(0,n))\\
                                         &=f(d_{\phi}(0,n))
\end{align*}
It follows that $\bar{\phi}$ is a morphism of $w(L)$-modules.
Conversely, let $\bar{\phi}:Hom_{w(L)}(\textbf{N}_{ab},A) $ be a
$w(L)$-homomorphism. We define $\phi: \textbf{g}\rtimes N\rightarrow
L\times_{f}\bar{A}$ given by
$$\phi(x,n):=\big(p(x),\bar{\phi}(n+<[N,N]>_{p})\big)$$
Then the associated derivation $d_{\phi}(x,n):=\bar{\phi}(n+<[N,N]>_{p})$ satisfies the relation $(2.1)$ and it follows that
$\phi \in Tor^{0}(p,M)$. Therefore the map
$\bar{\phi}\mapsto \phi$ is a well defined homomorphism, inverse to
the homomorphism $\phi \mapsto \bar{\phi}$.
\end{proof}

\begin{theorem}
Let $0\rightarrow N\rightarrow \mathbf{g} \rightarrow
\mathbf{b}\rightarrow 0$  be an exact sequence of restricted Lie
algebras and $A$ a $w(\mathbf{b})$-module. Then the following sequence is
exact
\end{theorem}
\[0 \rightarrow Der_{p}(\mathbf{b},A)\rightarrow Der_{p}(\mathbf{g},A)
\rightarrow Hom_{w(\mathbf{b})}(\mathbf{N}_{ab}, A)\rightarrow
H_{\textbf{G}}^{1}(\mathbf{b},A)\rightarrow
H_{\textbf{G}}^{1}(\mathbf{g},A)\]

\begin{proof} Since $\rm{RLie}$ is a monadic category over $\rm{Sets}$ whose objects have an underlying group structure, it follows from the above lemma
and the exact sequence $(1.5)$.
\end{proof}

\section{Internal groupoids and crossed modules}

Crossed modules in groups were introduced by Whitehead \cite{Whit}
in the study of relative homotopy groups. Brown and Spencer in \cite{BS} noted that internal categories within
the category of groups are equivalent to crossed modules. In the more
general context of categories of groups with operations, crossed modules
and internal categories are studied by Porter in \cite{Por}.
Moreover, internal categories in a Mal'tsev variety are studied by
Janelidze in \cite{Jane}. Bellow we consider the case of the category of restricted Lie algebras.

Let as recall the definition of internal category in a category
$\mathcal{C}$ with pullbacks. An \textit{internal category} in $\mathcal{C}$  is a
diagram in $\mathcal{C}$

\begin{equation*}
\xymatrix{
C\times_{C_{0}} C \ar@<1pt>[r]^-{\theta} & C \ar@<-4pt>[r]_-{t} \ar@<1pt>[r]^-{s} & C_{0} \ar@<-11pt>[l]_-{e}
}
\end{equation*}

such that $te=se=id_{C_{0}}$ with a morphism
$\theta: C \times_{C_{0}}C\rightarrow C$ where $C \times_{C_{0}} C$
denotes the pullback

\begin{align*}
 \begin{CD}
  C\times_{C_{0}} C @>p_{2}>> C \\
  @Vp_{1}VV  @VsVV \\
  C @>t>> C_{0} \\
 \end{CD}
\end{align*}

satisfying: $t\theta=tp_{2}$ and $s\theta=sp_{1}$,
the associative low relation

\begin{align*}
 \begin{CD}
  C \times_{C_{0}} C\times_{C_{0}} C @> \theta \times id >> C\times_{C_{0}}C \\
  @V id\times \theta VV  @V\theta VV \\
  C\times_{C_{0}} C  @>>\theta> C_{0} \\
 \end{CD}
\end{align*}

the left and right unit lows for composition of morphisms

$$
\xymatrix{ C \ar[rd]^-{id_{C}}
\ar[r]^-{( id_{C},et)} & C\times_{C_{0}} C
\ar[d]^-{\theta}
&C \ar[l]_-{(es, id_{C})} \ar[ld]^-{id_{C}}\\
&C}
$$

The morphisms $t,s,e$ are called the target, source and unit morphisms respectively. An internal category is called \textit{internal groupoid} if for any
$c\in \mathcal{C}$ there is a $c'\in \mathcal{C}$ such that
$\theta(c,c')=es(c)$ and $\theta(c',c)=et(c)$.

J.-L. Loday and C. Kassel define
in \cite{Loday2} the notion of crossed module for the category of
Lie algebras. In the same way are defined crossed modules for the
category of restricted Lie algebras (cf. \cite{dokas1}).

\begin{definition}
Let $\mu: M \rightarrow N$ be a homomorphism of restricted Lie
algebras. The triple $(M,N,\mu)$ is called a crossed module if there
is a restricted homomorphism $\eta: N \rightarrow Der(M)$ such that
$\eta (n)$ is a restricted derivation for all $n \in N$ and the
following relations hold:
\begin{align}
\mu( \eta (n)(m)) &=[n, \mu(m)],\;\; n \in N,\; m \in M \\
\eta (\mu (m))(m')&=[m,m'],\;\; m,m' \in M
\end{align}
\end{definition}
In order to simplify the notion we denote $\eta (n)(m):=n.m$ for all $m\in M$ and $n\in N$.
\begin{example}
Let $L$ be a restricted Lie algebra and $I$ an ideal of $L$. If $i:
I \hookrightarrow L$ denotes the inclusion homomorphism then the
$(I,L,i)$ is a crossed module in $RLie$.
\end{example}

\begin{theorem}
The notion of a crossed module is equivalent to the notion internal groupoid in the category of restricted Lie algebras.
\end{theorem}

\begin{proof}
Let $(M,N,\mu)$ be a crossed module in $\rm{RLie}$ then we associate to it the following diagram in $\rm{RLie}$

\begin{equation*}
\xymatrix{
 (M\rtimes N) \times_{N} (M\rtimes N) \ar[r]^-{\theta} & (M\rtimes N)  \ar@<-8pt>[r]_-{t} \ar@<-1pt>[r]^-{s} & N \ar@<-8pt>[l]_-{e}
}
\end{equation*}

where $s,t : M\rtimes N \rightarrow N$ are restricted Lie homomorphisms given by
$$s(m,n):=n$$

$$t(m,n):=n+\mu(m)$$ and $\sigma: N \rightarrow
M\rtimes N$ by $$e(n):=(0,n)$$ for all $m\in M$ and $ n\in
N$. The multiplication
 $$\theta: (M\rtimes N) \times_{N}
(M\rtimes N) \rightarrow M\rtimes N$$
is given by
$$\theta((m,n),(m',n+\mu(m))=(m+m',n)$$
A straightforward calculation shows that $\theta$ is a Lie algebra homomorphism. Moreover we note
that
\[
\begin{aligned}
\theta(((m,0),(0,\mu(m))^{[p]}) &=\theta((m,0)^{[p]},(0,\mu(m))^{[p]})\\
                                &=\theta((m^{[p]},0),(0,\mu(m)^{[p]}))\\
                                &=\theta((m^{[p]},0),(0,\mu(m^{[p]}))\\
                                &=(m^{[p]},0)\\
                                & =(\theta((m,0),(0,\mu(m)))^{[p]}
\end{aligned}
\]
and in the same way we see that
$$\theta(((0,n),(0,n))^{[p]})=(\theta((0,n),(0,n)))^{[p]}$$ and
$$\theta(((0,0),(m',0))^{[p]})=(\theta((0,0),(m',0)))^{[p]}$$

 Since $\theta$ is a Lie algebra homomorphism we deduce that is actually a
restricted Lie homomorphism. Then one can easily check that the conditions of associativity and unit and right lows are satisfied. Besides, if we define $$(m,n)':=(-m,n+\mu(m))$$ for all $m\in M$ and $n\in N$  then we see that the above diagram in $\rm{RLie}$ defines an internal groupoid in $\rm{RLie}$.

Let
\begin{equation*}
\xymatrix{
N'\times_{N} N'\ar@<1pt>[r]^-{\theta} & N'  \ar@<-4pt>[r]_-{t} \ar@<1pt>[r]^-{s} & N \ar@<-11pt>[l]_-{e}
}
\end{equation*}

be an internal groupoid in  $\rm{RLie}$ with multiplication $\theta$. Then
we associate to it a crossed module the $(M,N,\mu)$ where $M:=Ker\,
s$ and $\mu:=t\mid_{M}$ and with action $\eta: N \rightarrow Der(M)$
given by
$$\eta (n)(m)=[e(n),m]$$

for all $n \in N$ and $m\in M$. In effect,
\[
\begin{aligned}
\eta (n^{[p]})(m) &=[e(n^{[p]}),m]\\
                  &=[(e(n))^{[p]},m]\\
                  &=(\eta(n))^{p}(m)
                  \end{aligned}
\]
and
\[
\begin{aligned}
\eta(n)(m^{[p]}) &=[e(n),m^{[p]}]\\
                 &=ad^{p-1}(m)(\eta(n)(m))
\end{aligned}
\]

Since $te=se=id\mid_{N}$ we get
$$\theta(m+e(n),(m'+e(n+\mu(m))=\theta(m+e(n),e
t(m+e(n)))+\theta(e s(m'),m') $$

Besides by unities properties of the groupoid we have that
$$\theta ((m+e(n), et(m+e(n)))+\theta(e s(m'),m')=m+e(n)+m'$$

Thus we obtain
$$\theta((m+e(n)),(m'+e(n+\mu(m))=m+m'+e(n)$$

since $\theta$ is a Lie algebra homomorphism we have
$$\theta([(m,m'+e\mu(m)),(0,m'')])=[\theta(m,m'+e\mu(m)),\theta(0,m'')]$$
and
$$[e\mu(m),m'']+[m',m'']=[m,m'']+[m',m'']$$

We deduce that
\[
\begin{aligned}
 \eta (\mu(m)) &=[e(\mu(m),m'']\\
               & =[m,m'']
               \end{aligned}
               \]
Finally we have,
\[
\begin{aligned}
\mu(\eta (n)(m)) &=b([e(n),m])\\
         &=[te(n),t(m)]\\
         &=[n,\mu(m)]
\end{aligned}
\]
\end{proof}

\section{The second cohomology group and $2$-fold extensions}

Gerstenhaber in \cite{Ger} studies $2$-fold extensions in certain categories of interest including the category of Lie algebras. Also, the case of Lie algebras is studied by Shimada-Uehara-Brenneman-Iwai in \cite{Shi}. Besides, J.-L. Loday and C. Kassel in \cite{Loday2} consider two fold extensions of Lie algebras associated to a crosssed module. In this section we study $2$-fold extensions in the category of restricted Lie algebras.

Let $(M,N,\mu)$ be a crossed module in $\rm{RLie}$ and
$(A,[p_{A}]):=ker\,\mu$. By relation $3.2$ of Definition
$3.1$ we get that $\mu(a).m =[a,m]=0$ for all $a\in A$ and $m\in M$. Thus $A\subset \mathcal{C}(M)$ where $\mathcal{C}(M)$ denotes the
center of $M$. Moreover by relation $3.1$ of Definition $3.1$ we
have that $\mu(n.a)=[n,\mu(a)]=0$ thus $n.a \in A$ for all
$n\in N$ and $a\in A$.

The restriction on $A$ of the action of $N$ on $M$ endows $A$ with
the structure of restricted $N$-module. Moreover we have
\[
\begin{aligned}
n.a^{[p_{A}]} &=ad^{p-1}(a)(n.a)\\
              &=0\\
\end{aligned}
\]
thus the couple $(A,p_{A})$ is a Beck $N$-module. Besides
$\mu(n.m)=[n,\mu(m)]$ for all $n \in N$ and $m\in M$, thus we obtain that
$Img\,(\mu)$ is a restricted ideal of $N$. Since
\[
\begin{aligned}
(n+\mu(m)).a &=n.a+[m,a]\\
             &=n.a+0
\end{aligned}
\]
one sees that $A$ is also a restricted
$R$-module where $R:=coker\,\mu$. Therefore $(A,p_{A})$ becomes a
Beck $R$-module.

\subsection{2-fold extensions}

Let $(A,p_{A})$ be a Beck $R$-module. We denote by $E^{2}(R,A)$ the
category whose objects are exact sequences in $\rm{RLie}$

$$
0 \rightarrow A \rightarrow M \xrightarrow{\mu} \ N \rightarrow R
\rightarrow 0 $$ such that $(M,N,\mu)$ is a crossed module and the
induced Beck $R$-module structure is the given one. The morphisms
are:

\[
\begin{CD}
0 @>>> A @>>> M @>>> N @>>> R @>>> 0\\
  @.  @VVidV @VVfV @VVgV @VVidV\\
0 @>>> A @>>> M' @>>> N' @>>> R @>>> 0
\end{CD}
\]

where the morphisms $f$ and $g$ respect the actions. Two objects
$E_{1},E_{2} \in E^{2}(R,A)$ are called equivalent if there is a
morphism in $E^{2}(R,A)$ from one to the other. We will denote by
$\mathcal{E}^{2}(R,A)$ the set of equivalence classes in
$E^{2}(R,A)$.

For $p: E\rightarrow R$ a fixed epimorphism in $\rm{RLie}$ we
consider the set $E^{2}(p,A)$ the set of 2-term extensions
$$
0 \rightarrow A \rightarrow M \xrightarrow{\mu} \ N \xrightarrow{p}
R \rightarrow 0 $$ and morphisms

\[
\begin{CD}
0 @>>> A @>>> M @>>> N @>p>> R @>>> 0\\
  @.  @VVidV @VVfV @VVidV @VVidV\\
0 @>>> A @>>> M' @>>> N @>p>> R @>>> 0
\end{CD}
\]
Two objects $E_{1},E_{2} \in E^{2}(p,A)$ are called equivalent if
there is a morphism from one to the other. We denote by $\mathcal{E}^{1}(p,A)$
the set of equivalence classes.

\subsection{Baer sum} The set $E^{2}(R,A)$ can be endowed with the structure of abelian group. The Baer sum of two restricted Lie algebra extensions is induced from their Baer sum viewed as Lie algebra extensions. In particular, let  $E$ and $E'$ be $2$-fold extensions of restricted Lie algebras

$$(E):\;\;\;
0 \rightarrow A \rightarrow M \xrightarrow{\mu} \ N \rightarrow R
\rightarrow 0 $$

and

$$(E'):\;\;\;
0 \rightarrow A \rightarrow M' \xrightarrow{\mu'} \ N' \rightarrow R
\rightarrow 0 $$

The Baer sum of $(E)$ and $(E')$ is defined as the extension

$$(E+E'):\;\;\;
0 \rightarrow (A\times A)/K \rightarrow (M\times M')/K \xrightarrow \ N\times_{R}N' \rightarrow R
\rightarrow 0 $$

where $K:=ker\,f$ and $f: A\times A\rightarrow A$ given by $f(a,a'):=a+a'$ for all $a,a'\in A$. The restricted ideal $K$ is consisting of the elements $(a,-a)$.

Next we give an interpretation of $2$-torsors over a restricted Lie algebra $R$ in terms of $2$-fold extensions. We note that $2$-torsors relative to $U$ in $\rm{RLie}$ are $2$-torsors relative to $U$ in $\rm{Lie}$ viewed as simplicial objects in $\rm{Lie}$.

\begin{lemma}
Let $R \in \rm{RLie}$ be a restricted Lie algebra and $(A,p_{A})$ a
Beck $R$-module. Then we have
$$Tors^{2}(R,A)\simeq\mathcal{E}^{2}(R,A)$$
and
$$Tors^{1}(p,A)\simeq \mathcal{E}^{1}(p,A)$$
\end{lemma}

\begin{proof}
It suffices to construct maps from $Tors^{2}(R,A)$ to
$\mathcal{E}^{2}(R,A)$ inverse to each other. Let $(E.)$ be a
$2$-torsor over $R$ with simplicial morphism $$\epsilon_{.}: E_{.} \rightarrow K(A\times_{p_{A}} R,2)$$

 We consider the Moore complex $M(E.)$ of $(E.)$ given by $M(E.)_{0}=R$ and $M(E.)_{n}=\cap_{i}^{n}ker\,d_{i}$, for all $i\geq 1$ and with differential $\delta_{n}:=d_{0}|_{ M(E.)_{n}}$. By conditions $(1)$ and $(2)$ of the Definition $1.12$ the associated Moore complex $M(E.)$ is given by
the following exact sequence in $\rm{RLie}$

\[M(E.):\; 0 \rightarrow Ker\,d_{1}\cap \ker\,d_{0} \rightarrow ker\,d_{1} \xrightarrow{d_{0}} E_{0} \xrightarrow{p} R
\rightarrow 0 \]

It follows from condition $(2)$ of the Definition $1.12$ that $$E_{2}\simeq (A\times_{p_{A}} R)\times_R  \Lambda^{2}(1)(E_{.})$$
where $$ \Lambda^{2}(1)(E_{.})=\{(x_{0},x_{1})\in E_{1}\times E_{1}, d_{0}(x_{0})=d_{0}(x_{1})\}$$

We have the following commutative diagram

\[
\xymatrix{
E_{2} \ar@/_/[ddr]_{<d_{0},d_{1}>} \ar@/^/[drr]^{\epsilon_{2}}
\ar@{>}[dr]|-{\simeq} \\
& (A\times_{p_{A}} R)\times_R  \Lambda^{2}(1)(E_{.})\ar[d] \ar[r]
& A\times_{p_{A}} R \ar[d]_{pr_{R}} \\
& \Lambda^{2}(1)(E_{.}) \ar[r] & R }
\]

Therefore we obtain  $$Ker\,d_{1}\cap \ker\,d_{0} =(A,p_{A})$$

Since $2$-torsors in $\rm{RLie}$ are $2$-torsors in $\rm{Lie}$, it follows from \cite{Vale} that
the Lie action of $R$ on $A$ coincides with the Lie action induced from the above exact sequence. Moreover we have
\[
\begin{aligned}
r^{[p]}.a &=[s(r)^{[p]},a]\\
          &=[\underbrace{s(r)[s(r)[\cdots[s(r)}_{p},a]]]]\\
          &=(\underbrace{r.(r(\cdots (r}_{p}.a))))
          \end{aligned}
          \]
 where $s$ denotes a section of the surjection $p: E_{0}
\rightarrow R$. Therefore the induced Beck
$R$-module structure by the two-term exact sequence, coincides with
the initial structure and we obtain a two term exact sequence in $\rm{RLie}$

\[ 0 \rightarrow A \rightarrow Ker\,d_{1} \xrightarrow{d_{0}} E_{0} \xrightarrow{p} R
\rightarrow 0 \]

Conversely, let
$$0 \rightarrow A \rightarrow M \xrightarrow{\mu} E_{0} \xrightarrow{p} R \rightarrow 0 $$
be a two-term extension in $\mathcal{E}^{2}(R,A)$. By Theorem $3.3$ is associated an augmented over $R$ groupoid $\Gamma$:

\begin{equation*}
\xymatrix{
  E_{1}:=(E_{0}\rtimes M)  \ar@<-8pt>[r]_-{t} \ar@<-1pt>[r]^-{s} & E_{0} \ar@<-8pt>[l]_-{e} \ar@<1pt>[r]^{p}& R \ar@<1pt>[r] &0
}
\end{equation*}

where $s(e_{0},m):=e_{0}$, $t(e_{0},m):=\mu(m)+e_{0}$  and $e(e_{0}):=(0,e_{0})$ for all
$e_{0} \in E_{0}$ and $m \in M$. Let  $E_{.}:=cosk^{1}(\Gamma)$ be the $1$-coskeleton of $\Gamma$. If $$\Big((\mu(z_{1})+e_{0},x_{1}),(e_{0},y_{1}),(e_{0},z_{1})\Big) \in E_{2}$$ where $x_{1},y_{1},z_{1} \in M$ and $e_{0}\in E_{0}$ then $x_{1}+z_{1}-y_{1} \in A$. As in the case of Lie algebras is defined a Lie algebra homomorphism  $\delta_{2}: E_{2}\rightarrow K(A\times_{p_{A}} R,2)$ given by
$$\delta_{2}\Big((\mu(z_{1})+e_{0},x_{1}),(e_{0},y_{1}),(e_{0},z_{1})\Big)=(x_{1}+z_{1}-y_{1},p(e_{0}))$$

We notice that  $z_{1}+x_{1}=y_{1}+a$ for some $a\in A$, thus the Lie algebra $L_{(z_{1}+x_{1}),y_{1}}$ generated by  $(z_{1}+x_{1})$ and $y_{1}$ is zero. By Remark $1.2$ we obtain
\[
\begin{aligned}
(z_{1}+x_{1}-y_{1})^{[p]} &=(z_{1}+x_{1})^{[p]}-y_{1}^{[p]}+\sum_{i=1}^{i=p}s_{i}\big((z_{1}+x_{1}),y_{1}\big)\\
                          &=(z_{1}+x_{1})^{[p]}-y_{1}^{[p]}\\
                          &=z_{1}^{[p]}+x_{1}^{[p]}+\sum_{i=0}^{i=1}s_{i}(z_{1},x_{1})-y_{1}^{[p]}
\end{aligned}
\]
By Relation $(3.2)$  we have $\mu(z_{1}).x_{1}=[z_{1},x_{1}]$ thus

\[
\begin{aligned}
\delta_{2}\bigg(\Big((\mu(z_{1}),x_{1}),(0,y_{1}),(0,z_{1})\Big)^{[p]}\bigg)&=\delta_{2}\bigg(\Big((\mu(z_{1}),x_{1})^{[p]},(0,y_{1})^{[p]},(0,z_{1})^{[p]}\Big)\\
                                                                            &=\delta_{2}\bigg(\Big((\mu(z_{1})^{[p]},0)+(0,x_{1}^{[p]})+\sum_{i=1}^{i=p}
                                                                            s_{i}\big((\mu (z_{1}),0),(0,x_{1})\big)\Big),(0,y_{1}^{[p]}),(0,z_{1}^{[p]})\bigg)\\
                                                                            &=\delta_{2}\bigg(\Big(\mu(z_{1}^{[p]}),x_{1}^{[p]}+\sum_{i=1}^{i=p}
                                                                            s_{i}(z_{1},x_{1})\Big),(0,y_{1}^{[p]}),(0,z_{1}^{[p]})\bigg)\\
                                                                            &=(x_{1}^{[p]}+\sum_{i=1}^{i=p}s_{i}(z_{1},x_{1})+z_{1}^{[p]}-y_{1}^{[p]},0)\\
                                                                            &=\delta_{2}\Big((\mu(z_{1}),x_{1}),(0,y_{1}),(0,z_{1})\Big)^{[p]}
\end{aligned}
\]

Also, $p$ is a restricted Lie homomorphism so
$$\delta_{2}\bigg(\Big((e_{0},0),(e_{0},0),(e_{0},0)\Big)^{[p]}\bigg)=\delta_{2}\Big((e_{0},0),(e_{0},0),(e_{0},0)\Big)^{[p]}$$
Since $\delta_{2}$ is a Lie algebra homomorphism we see that  $\delta_{2}$ is actually a restricted Lie homomorphism. Besides, as for the case of Lie algebras
(see \cite{Vale})  $\delta$ is a normalized cocycle. It follows from Section $4.1$ in \cite{duskin} that $E_{.}$ is a $K(A\times_{p_{A}} R,2)$ torsor over $R$. Also, one can see that the Moore complex $M(E_{.})$ of $E_{.}$ is the exact sequence
$$0 \rightarrow A \rightarrow M \xrightarrow{\mu} E_{0} \xrightarrow{p} R \rightarrow 0 $$

Consequently, we defined maps from $$Tors^{2}(R,A) \rightarrow
E^{2}(R,A)$$ and $$E^{2}(R,A) \rightarrow Tors^{2}(R,A)$$ which are
inverse to each other. Moreover we can see that equivalent two-term
extensions correspond to the same connected component in the
category of $2$-torsors and conversely elements in the same
component correspond to equivalent two-term extensions.

Besides, by definition of the group structure on the set of torsors (see \cite{Glenn}, \cite{duskin2}), the group structure on $2$-torsors $RLie$ is induced from the group structure on $2$-torsors in $Lie$. Moreover, the Baer sum of two restricted Lie algebra extensions is induced from their Baer sum viewed as Lie algebra extensions. It follows that above bijections are actually group isomorphisms. Therefore the theorem follows.
\end{proof}

\begin{theorem}
Let $R \in \rm{RLie}$ be a restricted Lie algebra and $(A,p_{A})$ a
Beck $R$-module. Then there is an isomorphism
$$H^{2}_{\mathbf{G}}(R,A)\simeq \mathcal{E}^{2}(R,A)$$
\end{theorem}

\begin{proof}
The theorem follows from the above Lemma $4.1$ and the Theorem
$1.13$.
\end{proof}

\begin{remark}
We note that in the Cartan-Eilenberg context crossed modules in
various algebraic categories are associated with the third
cohomology group. In the Quillen-Barr-Beck context crossed modules
are associated to the second cohomology group since there is a shift
by 1 in the notation. Besides, G. Hochschild in \cite{Hoch2} gives an interpretation of the third Hochschild cohomology in terms of space of restricted kernel classes.
\end{remark}

\subsection{Eight term exact sequence}

The five term exact sequence of Theorem $2.2$ for Quillen-Barr-Beck
cohomology for restricted Lie algebras can be extended to an eight
term exact sequence by the following theorem.

\begin{theorem}
Let $0\rightarrow N\rightarrow \mathbf{g} \xrightarrow{p}
\mathbf{b}\rightarrow 0$  be an exact sequence of restricted Lie
algebras and $M$ an $w(\mathbf{b})$-module. Then the following sequence is
exact
\begin{multline*}
0 \rightarrow Der_{p}(\mathbf{b},M)\rightarrow Der_{p}(\mathbf{g},M)
\rightarrow Hom_{w(\mathbf{b})}(\mathbf{N}_{ab}, M)\rightarrow
H_{\textbf{G}}^{1}(\mathbf{b},M)\rightarrow
H_{\textbf{G}}^{1}(\mathbf{g},M) \\
\rightarrow \mathcal{E}^{1}(p,M)\rightarrow H^{2}_{\mathbf{G}}(\mathbf{b},M)\rightarrow
H^{2}_{\mathbf{G}}(\mathbf{b},M)
\end{multline*}
\end{theorem}
\begin{proof}
It follows from the Theorem $4.2$ and the eight term exact sequence
$(1.5)$.
\end{proof}

\subsection{Application to extensions of algebraic groups} Let $k$ be algebraic closed field of prime characteristic and $G$ an algebraic group over $k$.
 Since $char\;k=p$ we have that the associated Lie algebra $Lie(\rm{G})$ is actually a restricted Lie algebra. In fact in this way is defined a functor
 $Lie: \mathcal{G}r\rightarrow RLie$ from the category of algebraic groups to the
category of restricted Lie algebras (see \cite{Bor}).

 Let $H, A$ be algebraic groups then M. Rosenlicht in \cite{Ros} and J.-P. Serre in \cite{Ser} define as an extension of $H$ by $A$ a short exact sequence of groups
$$0\rightarrow A \xrightarrow{\alpha} G \xrightarrow{\gamma} H\rightarrow 0\;\;\; (E)$$ such that $\alpha,\gamma$ are separable rational homomorphisms. Thus $H$ can be identified with a normal algebraic subgroup of $G$ and $A$ with $G/A$.

Let $$0\rightarrow A \xrightarrow{\alpha'} G' \xrightarrow{\gamma'} H\rightarrow 0\;\;\;(E')$$ be an other extension of $H$ by $A$ then
$(E), (E')$ are called equivalent if there is a rational homomorphism $\psi: G\rightarrow G'$ such that the following diagram is commutative

\[
\begin{CD}
0 @>>> A @>>> G @>>> H @>>> 0 \\
  @.  @VVidV @VV\psi V @VVidV\\
0 @>>> A @>>> G' @>>> H @>>> 0
\end{CD}
\]

The class of equivalent extensions of $H$ by $A$ is denoted by $Ext(H,A)$.

\begin{proposition}
 Let $0\rightarrow H \rightarrow G \xrightarrow{\gamma}  A \rightarrow 0$ be an exact sequence of algebraic groups and $M$ an $w(Lie(A))$-module. Then the following sequence is exact
\begin{multline*}
0 \rightarrow Der_{p}(Lie(A),M)\rightarrow Der_{p}(Lie(G),M)
\rightarrow Hom_{w(Lie(A))}(\mathbf{Lie(H)}_{ab}, M)\rightarrow \\
\rightarrow H_{\textbf{G}}^{1}(Lie(A),M)\rightarrow
H_{\textbf{G}}^{1}(Lie(G),M)
\rightarrow \mathcal{E}^{1}(Lie(\gamma),M) \rightarrow \\
 \rightarrow H^{2}_{\mathbf{G}}(Lie(A),M)\rightarrow
H^{2}_{\mathbf{G}}(Lie(A),M)
\end{multline*}

\end{proposition}

 \begin{proof}
 If $0\rightarrow H \rightarrow G \xrightarrow{\gamma}  A \rightarrow 0$ is an exact sequence of algebraic groups
 then $0\rightarrow Lie(H) \rightarrow Lie(G) \xrightarrow{Lie(\gamma)} Lie(A) \rightarrow 0$ is an exact sequence of
 Lie algebras (see \cite{Ros}). Since $k$ is a field of characteristic of $p$ the induced Lie algebra exact sequence is
 actually a restricted Lie algebras sequence. Therefore the proof follows from the Theorem $4.4$ above.
\end{proof}

 If $0\rightarrow H \rightarrow G \xrightarrow{\gamma}  A\rightarrow 0$ is an exact sequence of commutative algebraic groups and $B$ a commutative algebraic group, then J.-P. Serre in \cite{Ser} and M. Rosenlicht in \cite{Ros} proved that there is an exact an exact sequence
\begin{multline*}
0\rightarrow Hom(A,B) \rightarrow Hom(G,B) \rightarrow Hom(H,B)\rightarrow\\
\rightarrow Ext(A,B) \rightarrow Ext(G,B)\rightarrow Ext(H,B)
\end{multline*}

Since $0\rightarrow H \rightarrow G \xrightarrow{\gamma}  A\rightarrow 0$ is an exact sequence of commutative algebraic groups it is induced an exact sequence of abelian restricted Lie algebras $0\rightarrow Lie(H) \rightarrow Lie(G) \xrightarrow{Lie(\gamma)}  Lie(A) \rightarrow 0$. Then $Lie(B)$ becomes a $w(Lie(A))$-module when $Lie(A)$ acts trivially and  $R_{f}$  acts via $fb:=b^{[p]}$ for all $b\in Lie(B)$. Thus we get an exact sequence

\begin{multline*}
0 \rightarrow Hom_{RLie}(Lie(A), Lie(B))\rightarrow Hom_{RLie}(Lie(G),Lie(B))
\rightarrow Hom_{RLie}(Lie(H),Lie(B))\rightarrow \\
\rightarrow H_{\textbf{G}}^{1}(Lie(A),Lie(B))\rightarrow
H_{\textbf{G}}^{1}(Lie(G),Lie(B))
\rightarrow \mathcal{E}^{1}(Lie(\gamma),Lie(B)) \rightarrow\\
 \rightarrow H^{2}_{\mathbf{G}}(Lie(A),Lie(B))\rightarrow
H^{2}_{\mathbf{G}}(Lie(A),Lie(B))
\end{multline*}

\bibliographystyle{plain}

\begin{thebibliography}{10}

\bibitem {B} M. Barr, \textit{Cartan-Eilenberg cohomology and Triples,}
Journal of Pure and Applied Algebra \textbf{112} (1996), 219-238.

\bibitem {BB} M. Barr, J. Beck, \textit{Homology and standard constructions,
    Seminar on Triples and Categorical Homology Theory,} vol. 80,
  Springer-Verlag, 1969, pp. 245-335.

\bibitem{BarRin} M. Barr, G. Rinehart, \textit{Cohomology as the derived functor of derivations}, Trans. Amer. Math. Soc. \textbf{122} (1966) 416–426.

\bibitem{Bor} A. Borel, \textit{Linear algebraic groups}, Graduate Texts in Mathematics, Springer-Verlag, New York,
1991.

\bibitem{BS} R. Brown and C. B. Spencer, \textit{G-groupoids, crossed modules and the fundamental
groupoid of a topological group}, Proc. Kon. Ned. Acad. \textbf{79}
(1976), 296-302.

\bibitem{CE} H. Cartan and S. Eilenberg, \textit{Homological Algebra}, Princeton University Press, Princeton, 1956.

\bibitem {CegAz} A.M. Cegarra, E. R Aznar, \textit{ An exact
sequence in the first variable for torsor cohomology: The
$2$-dimensional theory of Obstructions}, \textbf{39} (1986), 197-250.

\bibitem {Ceg2}A.M. Cegarra, M. Bullejos and A.R. Garz, {Higher Dimensional obstruction theory in algebraic
categories}, Journal of Pure and Applied Algebra 49 (1987) 43-102.


\bibitem {dokas1} I. Dokas, \textit{A (co-) homology for restricted Lie algebras}, Thesis, University of  Warwick, 2000.

\bibitem {dokas} I. Dokas, \textit{Quillen-Barr-Beck (co-) homology for
    restricted Lie algebras,} Journal of Pure and Applied Algebra, \textbf{186} (2004),
    33-42.

\bibitem {duskin} J. Duskin, \textit{Simplicial methods and the interpretation
    of Triple Cohomology}, Mem. Amer. Math. Soc. issue 2, \textbf{163}, 1975.

\bibitem {duskin2} J. Duskin, \textit{Higher-dimensional torsors and the cohomology of topoi: the abelian theory}, Applications of sheaves
(Proc. Res. Sympos. Appl. Sheaf Theory to Logic, Algebra and Anal.,
Univ. Durham, Durham, 1977), pp. 255-279, Lecture Notes in Math.,
753, Springer, Berlin, 1979.

\bibitem{ES} B. Eckmann, U. Stammbach, \textit{On exact sequences in the homology
of groups and algebras}, Illinois J. Math. \textbf{14} (1970),
205–215.

\bibitem{Far} H. Strade, R. Farnsteiner, \text{Modular Lie algebras and their representations}, Marcel Dekker, 1988.

\bibitem {Ger} M. Gerstenhaber, \textit{On the deformation of rings and
algebras} II. Ann. of Math. \text{84} (1966), 1–19.

\bibitem{Glenn} P. G. Glenn, \textit{Realization of cohomology classes in
    arbitrary exact Categories,} Journal of Pure and Applied Algebra
  25, (1982), 33-105.
\bibitem{Hoch} G. Hochschild, \textit{Cohomology of restricted Lie algebras}, Amer. J.
Math. \textbf{76}, (1954), 555-580.

\bibitem{Hoch2} G. Hochschild, \textit{Lie algebra kernels and cohomology}
Amer. J. Math. 76, (1954), 698–716.

\bibitem {Jac} N. Jacobson, \textit{Lie algebras}, Dover Publications, Inc, 1979.

\bibitem {Jane} G. Janelidze, M.C Pedicchio, \textit{Internal Categories and
  Groupoids in Congruence Modular Varieties}, Journal of Algebra, \textbf{193} (1997), 552-570.


\bibitem {Loday2} C. Kassel, J.-L. Loday, \textit{Extensions centrales d' algebres de Lie},
  Ann. Inst. Fourier \textbf{32} (1982), 119-142.

\bibitem {Loday} J.-L. Loday, \textit{Spaces with finitely many non-trivial homotopy groups,} Journal of Pure and Applied Algebra, \textbf{24} (1982), 179-202.

\bibitem {Pa} B. Pareigis, \textit{Cohomologie des p-alg\`{e}bres de Lie}, C. R. Acad. Sci. Paris Ser. A-B \textbf{263}, (1966), 709-712.

\bibitem {Por} T. Porter, \textit{Crossed modules in Cat and a Brown-Spencer
theorem for 2-categories}, Cahiers Topologie Géom. Différentielle
Catég. \textbf{26} (1985), 381-388.

\bibitem{Ros} M. Rosenlicht, \textit{Extensions of vector groups by abelian varieties}, Amer. J. Math. \textbf{80}, (1958), 685–714.

\bibitem{Vale} M. Vale, \textit{Torsors and special extensions}, Cahiers Topologie Geom. Differentielle Categ. \textbf{26} (1985), 63-90.

\bibitem{Wat} W. C. Waterhouse, \textrm{Introduction to Affine Group Schemes}, Springer Verlag, 1979.

\bibitem{Ser} J.-P. Serre, \textit{Quelques proprietes des varietes abeliennes en caracteristique p}, Amer. J. Math. \textbf{80} (1958) 715–739.

\bibitem {Shi} N. Shimada, H. Uehara, F. Brenneman, A. Iwai, \textit{Triple cohomology of algebras and two term extensions}, Publ.
Res. Inst. Math. Sci. \textbf{5} (1969) 267-285.

\bibitem {Quillen} D. Quillen, \textit{On the (co-) homology of commutative
  rings, Proceedings of Symposium on Pure Mathematics}, Vol XVII, AMS,
  Providence, RI, 1970.

\bibitem {Quillen2} D. Quillen, \textit{Homotopical algebra}, Lecture
Notes in Mathematics, No. 43 Springer-Verlag, Berlin-New York 1967.



\bibitem {Whit} J. H. C. Whitehead, \textit{Combinatorial homotopy}, II, Bull. Amer. Math.
Soc. \textbf{5} (1949), 453-496.


\end{thebibliography}

\end{document}